\documentclass[oneside,12pt]{amsart}
\usepackage {CJKutf8}
\usepackage{amssymb}
\usepackage{amsthm}
\usepackage[abbrev]{amsrefs}
\usepackage{hyperref}

\usepackage{enumitem}

\usepackage{todonotes}

\usepackage{float} % to fix tables

\usepackage{pgfplots}

\hypersetup{colorlinks=true,allcolors=[rgb]{0,0,0.6}}

\usepackage{tikz,pgflibraryshapes}
\usetikzlibrary{3d,arrows.meta,decorations.pathreplacing}

\usepackage{enumitem}

\newcommand{\vertiii}[1]{{\left\vert\kern-0.25ex\left\vert\kern-0.25ex\left\vert #1 
		\right\vert\kern-0.25ex\right\vert\kern-0.25ex\right\vert}}

\theoremstyle{plain}
\newtheorem{theorem}{Theorem}
\newtheorem{proposition}[theorem]{Proposition}
\newtheorem{lemma}[theorem]{Lemma}
\newtheorem{corollary}[theorem]{Corollary}

\theoremstyle{remark}

\newtheorem{remark}{Remark}

\theoremstyle{definition}

\numberwithin{equation}{section}
\numberwithin{theorem}{section}

\usepackage[normalem]{ulem}
\usepackage{soul}
\setstcolor{blue}

\definecolor{green}{rgb}{0.0, 0.5, 0.5}

\definecolor{lgray}{gray}{0.9}
\definecolor{llgray}{gray}{0.95}
\definecolor{lllgray}{gray}{0.975}

\newcommand{\R}{\mathbb{R}}

\newcommand{\cB}{\mathcal{B}}
\newcommand{\cD}{\mathcal{D}}

\newcommand{\cX}{\mathcal{X}}

\newcommand{\nc}{\newcommand}

\nc{\al}{\alpha}
\nc{\h}{\delta}
\nc{\G}{\Gamma}
\nc{\et}{\eta} 
\nc{\g}{\gamma}
\nc{\gam}{\gamma}
\nc{\ka}{\kappa}
\nc{\lam}{\lambda}

\nc{\Lam}{\Lambda}
\nc{\Om}{\Omega}
\nc{\om}{\omega}

\nc{\ta}{\tau}
\nc{\w}{\omega}
\nc{\io}{\iota}
\nc{\z}{\zeta}
\nc{\s}{\sigma}
\nc{\si}{\sigma}
\nc{\Si}{\Sigma}
\nc{\vphi}{\varphi}

\newcommand{\eps}{\epsilon}

\renewcommand{\k}{\kappa}

\nc{\ran}{\rangle}
\nc{\lan}{\langle}

\newcommand{\one}{\mathbf{1}}

\newcommand{\supp}{\operatorname{supp}}

\newcommand{\dist}{\mathrm{dist}}

\nc{\bfone}{{\bf 1}}
\nc{\1}{{\bf 1}}

\nc{\dd}{\mathrm{d}}

\renewcommand{\div}{\operatorname{div}}

\newcommand{\DETAILS}[1]{}

\newcommand{\abs}[1]{\ensuremath{\lvert#1\rvert}}
\newcommand{\norm}[1]{\ensuremath{\left\lVert#1\right\rVert}}
\newcommand{\sbr}[1]{\left[#1\right]}

\newcommand{\del}[1]{\Bigl(#1\Bigr)}
\newcommand{\thmref}[1]{Theorem~\ref{#1}}

\newcommand{\secref}[1]{Section~\ref{#1}}
\newcommand{\lemref}[1]{Lemma~\ref{#1}}
\newcommand{\propref}[1]{Proposition~\ref{#1}}
\newcommand{\remref}[1]{Remark~\ref{#1}}
\newcommand{\figref}[1]{Figure~\ref{#1}}
\newcommand{\corref}[1]{Corollary~\ref{#1}}

\newcommand{\inn}[2]{\left\langle#1,\,#2\right\rangle}

\newcommand{\grad}{\nabla}
\newcommand{\Lap}{\Delta}
\newcommand{\di}{\partial}

\newcommand{\Rb}{\mathbb{R}}

\newcommand{\cp}{\mathrm{c}}

\newcommand{\es}{\operatorname*{ess\,sup}}

\DeclareMathOperator{\inj}{inj}
\DeclareMathOperator{\Ric}{Ric}

\usepackage{amsthm}

\theoremstyle{definition}

\renewcommand{\bar}{\overline}

\author{Marius Lemm}
\address{Department of Mathematics, University of T\"ubingen, 72076 T\"ubingen, Germany }
\email{marius.lemm@uni-tuebingen.de}

\author{Israel Michael Sigal} \address{Department of Mathematics, University of Toronto, Toronto, M5S 2E4, Ontario, Canada} \email{im.sigal@utoronto.ca}

\author{Jingxuan Zhang} \address{Yau Mathematical Sciences Center, Tsinghua University, Haidian District, Beijing 100084, China} \email{jingxuan@tsinghua.edu.cn}

\date{\today}


\begin{document}
	
	\title[Diffusion bounds]{Diffusion bounds for non-autonomous degenerate parabolic equations
	}

	\maketitle
	\bibliographystyle{abbrv}
	\begin{abstract}
 {We consider possibly degenerate non-autonomous  parabolic equations in $L^p$ spaces for all $1\le p\le\infty$, and prove Gaussian space-time  upper bounds of  the Davies-Gaffney, or integrated Nash-Aronson, type on their solutions. Our bounds cover both linear and nonlinear equations.}

	\end{abstract} 
	
	\section{Introduction}

	Diffusion bounds quantify how solutions to parabolic equations 
	\begin{align}
		\di_t u =L u,\label{11b}
	\end{align}
	where $L$ is an elliptic operator (generator), 
	 spread in space over time. In the classical uniformly elliptic setting, they are encapsulated by pointwise Gaussian heat-kernel estimates of Nash--Aronson type \cite{Aro,Nas} and their integrated variants due to Davies and Gaffney  \cite{Dav,Davb,Davc,Gaf}. Investigating them under relaxed assumptions on the generator and on the geometry  of the underlying spaces constitutes a fundamental and rich  subject in the PDE and geometry literature, see  \cite{Sala,Salb,Gri,GH,GJK,GHHb,GHH,grigor1994integral,grigor1994heat,cheng1981upper,LY,AN,MS,MSa} for some contributions.  
	 
	 Of many important results, we mention that Li-Yau \cite{LY} gave the first extension to Riemannian manifolds with curvature bounds by leveraging gradient estimates building on earlier work with Cheng \cite{cheng1981upper}. The curvature assumption was later weakened by Saloff-Coste \cite{Sala} and more general Riemannian manifolds were covered by Grigor'yan  by introducing new ideas based on Faber-Krahn inequalities \cite{grigor1994heat} and the maximum principle \cite{grigor1994integral}; see also \cite{Gri,grigoryan2009heat}. A recent focus has been on the extension to metric spaces \cite{GH,GJK} and to jump-type Dirichlet forms \cite{GHH}   for autonomous linear equations. We note that most of the existing literature relies on uniform ellipticity of the generator or at least some form of hypoellipticity \cite{Dav}.\\

	In this paper, we develop  space-time off-diagonal upper bounds of Davies-Gaffney type for a class of parabolic problems, which includes  degenerate, non-autonomous   ones, covering both linear and non-linear equations. 
	Concretely, we consider the evolution equation
	\begin{align}
		\label{PE}
		\di_t u  = L u,\quad \quad u:\Om\times I\to \Rb.
	\end{align}
	Here, $I$ is a time interval, $\Omega$ is either  the Euclidean space $\Rb^n$, or a   domain in a  smooth Riemannian manifold (possibly with boundary), and  $L$ is a non-autonomous divergence-form linear operator (generator),
	\begin{align}
		\label{Lgen}
		L u = \div (a  \, \grad u ) + \inn{ b}{\grad u}  +c u,
	\end{align}
	where $\inn{\cdot}{\cdot}$ is either the Euclidean inner product or a Riemannian metric on $\Om$, and the coefficients $a,\,b,$ and $c$ are measurable functions of $x\in\Om$ and $t\in I$.
	
	 We require  mild assumptions on $a,b,c$,
	substantially weaker  than ellipticity or hypoellipticity, e.g.
	 the matrix-valued function $a$ is allowed to vanish  {on arbitrary subsets of $\Om$}. Since the usual parabolic solution theory breaks down in this setting, we work with weak solutions throughout.

	Our main results are $L^p$-upper bounds on weak solutions to \eqref{PE} for all $1\le p \le \infty$. For $p=2$ and a large class of $\Om$'s, they generalize the Davies-Gaffney ones. Given $Y\subset \Omega$,  consider an initial data with $\supp u_0\subset Y$. Given another subset $X\subset \Omega$ disjoint from $Y$, we prove (see \thmref{thm})
	\begin{align}\label{eq:intro_result}
		\norm{u(t)}_{L^p(X) } \le   \exp\del{-\frac{(d_{XY}-\k't)^2}{4 \al'  t}}\norm{u_0}_{L^p(Y)},\qquad  0< t< d_{XY}/\k'.
	\end{align}
	Here,
	 $\alpha'>0$ is a  constant depending  on the coefficients of $L$, $d_{XY}$ is the (geodesic) distance between $X$ and $Y$,  and $\k'>0$ is a constant depending  on the background geometry.

		The bound \eqref{eq:intro_result} establishes that outside the ballistic region  $ t\in (0,   {d_{XY}}/\k')  $, the solution $u(t)$ decays \textit{diffusively}, i.e., as a function of  $\frac{d_{XY}^2}{t}$ at an exponential rate.   {The ballistic validity interval} is necessary, as the   drift and $\grad a$ components could lead to a traveling-wave like ballistic dynamics  (see   \remref{rem22}\textup(iii)). The same could occur in nonlinear settings e.g., for geometric reasons like negative curvature; see \cite{davies1988heat} and \cite[Appendix A]{lemm2018heat}. Existing diffusion  bounds of Nash-Aronson type often also allow for a drift component and leave the validity interval implicit in the growth of the constant prefactor (\cite{Aro,Nas}). 

		We  establish  \eqref{eq:intro_result} under weak assumptions that cover a broad class of parabolic equations on a large class of Riemannian manifolds (Theorem \ref{thm1}).	 Crucially, it only requires   $a(\cdot,t)\geq 0$. 
		
		To compare our findings with earlier results, 
{Ataei and Nystr\"om proved the existence of a fundamental solution and Gaussian bounds (\cite{AN}), when the degeneracy is dictated by a spatial $A_2$-weight. The assumption in \cite{AN} thus only allows $a$ to vanish on null sets.
Heat kernel estimates from  weighted $L^1\to L^1$ propagator estimates have previously been studied in \cite{MS,MSa} for autonomous operators satisfying a Sobolev embedding inequality. This assumption thus  also only allows $a$  to vanish on null sets.
 The $L^2$-off diagonal bound for  degenerate elliptic operators was proved by ter Elst et al in \cite{tERSZa} for  autonomous,  zero-drift  $L$, using   a viscosity method (ours deals directly with the degenerate propagator). 
Extensions to higher-order elliptic operators were obtained in \cite{Dava,Bar,BB}. }

	Our approach is relatively simple and robust.  {Like the classical Davies' method (\cite{Dava,Davc}), it relies on the exponential deformation (twisting) of the generator $L$. We extend Davies' method to a combination of time-dependent, strongly degenerate operators, and a common Gaussian exponent for all \(1\le p\le\infty\), as well as a conditional nonlinear version. }

Our proof does not rely on any tools related to ellipticity, such as Li-Yau-type gradient estimates \cite{LY}, the maximum principle \cite{grigor1994integral},  Sobolev inequalities \cite{Dav,Davd}, or the spectral gap of $L$ (manifesting, e.g., through Faber-Krahn type inequalities) \cite{grigor1994heat}, or hypoellipticity condition (\cite{S_n}). The method extends to nonlinear degenerate parabolic equations  (e.g., the~McKean-Vlasov equation); see Section \ref{sec:examples}.

	\subsection*{Organization of the paper}
	In \textit{Section \ref{sec2}}, we introduce the setup and the basic solution theory we work with. Afterwards, we state our main results, Theorems~\ref{thm} and \ref{thm1}, on diffusion bounds in the Euclidean and Riemannian setting, respectively, as well as Theorem \ref{thm2} about the extension to nonlinear PDE. In  \textit{Section \ref{secPfThm1}}, we develop the overall proof strategy and prove Theorem~\ref{thm1}. In \textit{Section \ref{sec42}}, we describe the necessary modification to obtain the nonlinear result in Theorem  \ref{thm2}.  In \textit{Section \ref{sec:examples}}, we discuss applications of our bounds  to examples of nonlinear PDE, such as the  McKean-Vlasov equations.

	\subsection*{Notation}
	Given a smooth domain $\Om$ in a Riemannian manifold $(M,g)$,
	we denote by $d(x,y)$   the geodesic distance between $x,\,y\in \Om$ w.r.t.~$g$, and   $\dist(X,Y)=\inf _{x\in X,y\in Y}d(x,y)$ the distance between $X,\,Y\subset\Om$. For $r>0$ and $z\in \Om$,  $B_r(z)$ denotes the geodesic ball of radius $r$ around $z$.
	For $S\subset \Om$, we  denote by $\overline{S}$  the closure of $S$ in $\Om$,  and 
	$S^\cp = \Omega \setminus S$   the complement of $S$ in $\Om$. 	
	{For a matrix-valued function $a$, we write $\norm{a}_{L^\infty}=\es\norm{a(x)}_{\mathrm{op}}$, and  $ \div a$ stands for the row divergence with $j$-th entry given by $ (\div a )^j= \sum _i \di_i a^{ij}$.} For a  Banach space $\cX$, denote by $\cB(\cX)$ the class of bounded operators on $\cX$. 
	When no confusion arises, we abbreviate $\norm{\cdot}_p\equiv \norm{\cdot }_{L^p}$, and $\norm{\cdot}_{p\to p}$ the operator norm on $\cB(L^p)$,  for $1\le p\le \infty$.

	\section{Setup and results}
	\label{sec2}

	\subsection{Cauchy problem}
 Let $I=(0,\infty) $, and let  $\Om$ be either $\Rb^n$, or  a relatively compact domain with a smooth boundary in an  $n$-dimensional smooth connected Riemannian manifold $(M,g)$ (not necessarily complete), cf.~Grigor'yan (\cite{Gri}).

  We study  weak  solutions 
 to the Cauchy problem with initial condition $u_0\in L^p,$ 
 \begin{align}
 	\label{PE1}
 	\begin{cases}
 		\partial_t u = L u, & \text{in } \Omega\times I,\\
 		u(\cdot, 0) = u_0, &
 	\end{cases}
 \end{align}with $L$ given by \eqref{Lgen},
 subject to the homogeneous Neumann boundary condition 
 \begin{align}
 	\label{BC}
 	\inn{a\grad u }{\nu}\vert_{\di\Om}=0 \quad \text{if $\di\Om\ne\emptyset$.}
 \end{align}

In the expression \eqref{Lgen} of the operator $L$,   in the Riemannian case, $\grad$ and $\div$   are the Riemannian gradient and divergence   associated to~the metric $g$, respectively.
	For each fixed $t$, the coefficients of  $L=L(t)$ consist of a symmetric tensor field
	$a(\cdot, t)$, a vector field $b(\cdot, t)$, and a function $c(\cdot,t)$ on $\Om$.

We assume  the coefficients of $L$ satisfy
\begin{align}
	\label{aCond}
	\begin{cases}
			& a,\,b\in L^\infty(I,W^{1,\infty} (\Omega)),\quad c\in L^\infty(I,L^\infty(\Om)),\\
	&a \ge 0, \quad  {\div b  -c \ge 0},\quad   c  \le0,\quad \text{a.e.~in }\Om\times I,\\
	&  \inn{b }{\nu} \vert_{\di\Om}=0\quad \text{if $\di\Om\ne\emptyset$,}
	\end{cases}
\end{align}
	where $\nu$ denotes the outward unit normal vector on $\di\Om$. 
	
Our setup includes the case $a\equiv 0$, in which one has a transport equation and standard parabolic solution theory (e.g., energy estimates) breaks down.
	Therefore, we work with \textit{weak} solutions in $L^p$-sense, as is common for transport equations \cite{GG,DL}. 	 
		Given $u_0\in L^p(\Om)$, we say $u$ is  an $L^p$-weak solution to \eqref{PE1}--\eqref{BC}, if $u\in L^\infty(I,L^p(\Om))$ and
		\begin{align}
			\label{DS}
			\int_{I} \int_\Om u (\di_t+L^*) \varphi\,dx\,dt+\int_\Om u_0 \varphi(\cdot,0)\,dx =0, 
		\end{align}
		for all $ \varphi\in C_c^\infty(\overline \Om\times [0,\infty))$ with $\inn{a\grad \varphi }{\nu}\vert_{\di\Om}=0$ if $\di\Om\ne\emptyset$, where  
		$$
		L^* \varphi=\operatorname{div}\left(a\nabla \varphi\right)-\div(b\varphi)+c \varphi 
		$$
		is the formal adjoint of $L$.

The regularity conditions in \eqref{aCond} ensure the Cauchy problem
\eqref{PE1}--\eqref{BC} admits a solution in $L^p$, $ 1\le p\le \infty$, in the following sense:
\begin{proposition}[Existence of propagator]\label{prop1}Let \eqref{aCond} hold. Then 
there exists 
	a family of operators $P_t\in \cB(L^2(\Om))$, $t\ge0 ,$  s.th.
\begin{enumerate}[label=\textup{(P\arabic*)},ref=\textup{(P\arabic*)}]
	\item\label{P1}  {$t\mapsto P_t$ is continuous in the weak operator topology in $\cB(L^2)$;}
	\item\label{P2} for each $t$, $P_t$ is positivity-preserving and extends to a contraction on $L^p$, $1\le p\le \infty$;
	\item\label{P3} for each $u_0\in L^p(\Om)$, $u(t)=P_tu_0$ is an $L^p$-weak solution to \eqref{PE1}--\eqref{BC}.
\end{enumerate}
\end{proposition}
 {This proposition is proved in  Appendix \ref{secWP}.	We call $P_t$ as in \propref{prop1} a propagator generated by $L$. In the remainder of this paper, $P_t$ denotes the propagator furnished by \propref{prop1}, obtained explicitly  by \eqref{eq:luwot}. }

 \subsection{Main results}

 Within this section, we set the time interval $I=(0,\infty)$, and write $d_{XY}=\dist(X,Y)$.

	\begin{theorem}[Diffusion estimate in $\Rb^n$] \label{thm}

	Assume  there exist $\al,\,\beta>0$ s.th.
	\begin{align}
	&	\norm{a}_{L^\infty(\Rb^n\times I)}\le \al,\quad
	 \norm{b}_{L^\infty(\Rb^n\times I)}+\norm{\div a}_{L^\infty(\Rb^n\times I)} \le\beta. \label{27}
	\end{align}
		Then, for any $\delta>0$, 
	there exist $\k=\k(n,\al,\beta,\delta)$
	and $\g=\g(n)$ s.th.~for any  disjoint non-empty measurable subsets $X,\,Y\subset \Rb^n,$
	\begin{align}\label{mainSp}
		\norm{\chi_X P_t\chi_Y}_{L^p(\Rb^n)\to L^p(\Rb^n)} \le   \exp\del{-\frac{( d_{XY}-\k' t)^2}{ 4 \al 't}},
	\end{align}	
	for all $1\le p\le \infty$ and $0< t <   d_{XY} /\k'$, provided  $d_{XY}\ge\delta$, where $\al'=\al/\g^2$ and $\k'=\k/\g$. 
	\end{theorem}
	 {This theorem is proved shortly as a consequence of a more general result.}
			\begin{remark}  \label{rem22}
		\leavevmode

		\begin{enumerate}[label=(\roman*)]

\item 		 {The results here and below extend to any initial time $s\in\Rb$ by time translation, with $P_t$ replaced by $P_{t,s}$, $t$ in the estimates replaced by $t-s$, and the coefficient assumptions imposed on $(s,\infty)$.}
			\item	For $p=1$, the condition $c\le 0$ in \eqref{aCond} can be dropped.

			\item 	The ballistic validity interval is  required and optimal, since both $b$ and $\div a$ can generate ballistic motion. 
			This can be seen by considering the Laplacian with a constant drift $\Lap - b_0\cdot \grad$ with $\k= |b_0 |$. 
			
			Even when $b=0$, we still have the ballistic validity interval due to local ballistic motion induced by $\grad a$. 
			E.g., consider %
			\begin{align}
				\label{te}
				\di_t u=\di_x(a \di_x u),\quad u :\Rb\times (0,\infty)\to \Rb,
			\end{align} where $a(x,t)=A(x-\g t)$ for some $0<\g<1$ and, for some $R>0$,  the function  \begin{align}
				A(\mu):=\begin{cases}
					0,&\mu\le 0,\\ \mu,&0<\mu<R,\\R,&\mu\ge R.
				\end{cases}
			\end{align} Then conditions \eqref{aCond} hold, and the equation \eqref{te} has the  traveling wave solution   $u(x,t)=\varphi(x-\g t)$ in the weak sense, where  
			\begin{align}
				\varphi(\mu)=\begin{cases}0,& \mu\le0,\\
					\mu^{-\g},&0<\mu< R,\\
					e^{\g} R^{-\g} e^{-\g \mu / R}, & \mu \ge R.
				\end{cases}
			\end{align}
			Note also that $\varphi\in L^p$ for  all $1\le p< 1/\g$, and $\di_x a = 1$ near the ballistic wave front $x\sim \g t$.
		\end{enumerate}
	\end{remark}

\thmref{thm} is a special case of a  more general result as follows. 
 Let $C^\infty_b(\bar\Om)$ denote the class of bounded real-valued functions on $\bar \Om$ that is smooth in $\Om$ up to the boundary, and   satisfies $  \grad\psi \vert_{\di\Om}=0$ if $\di\Om\ne\emptyset$. Given $\k,\,\g>0$ and disjoint non-empty  subsets $X,\,Y\subset \Om$, we write \begin{align}\label{psiXY}
	\psi(X,Y):=\inf _X\psi  - \sup _Y \psi,
\end{align} and we say $X,\,Y$ are $(\k,\g)$-separated  if there exists $\psi\in C_b^\infty(\bar \Om)$ s.th.
\begin{align}
	&	\es_{\Om\times I}\abs{\inn{a \grad\psi}{\grad\psi}}\le1,\quad	 \es_{\Om\times I}\del{\abs{\div(a\grad\psi)}+\abs{\inn{b}{\grad\psi}}}\le \label{kCond}
	\k,\\&\psi(X,Y)\ge \g d_{XY}>0.\label{gCond}
\end{align}
\thmref{thm} follows from the following $L^p\to L^p$ diffusion estimate for $P_t$:
		\begin{theorem}[Diffusion estimate, general case] \label{thm1}
		For any $\k,\,\g>0$ and  non-empty measurable $(\k,\g)$-separated subsets $X,\,Y\subset \Om$, we have, for  all   $0< t <  \g d_{XY} /\k$,  $1\le p\le \infty$, 
		\begin{align}
			\label{122}
			\norm{\chi_X P_t\chi_Y}_{L^p(\Om)\to L^p(\Om)} \le   \exp\del{-\frac{(\g d_{XY}-\k t)^2}{ 4  t}}.
		\end{align}	
	\end{theorem}
	This theorem is proved in \secref{secPfThm1}.

	 {Informally, conditions \eqref{kCond}--\eqref{gCond} ask for a bounded smooth function $\psi$ that acts as a well-behaved distance separating $X$ from $Y$. Condition \eqref{gCond} requires  $\psi$ to be comparable to the geodesic distance, while the two bounds in \eqref{kCond} control, respectively, the slope of $\psi$ in the diffusive directions and its variation under the diffusion and drift. Such separation functions are standard in geometric analysis. In $\Rb^n$, they can be obtained by smoothly truncating the distance function to one of $X,\,Y$. On complete manifolds, Schoen--Yau constructed distance-like exhaustion functions with controlled gradient and Laplacian under a lower Ricci-curvature bound \cite[Thm.~4.2]{SY}; such functions with controlled Hessian are available under additional geometric assumptions \cite[Prop.~1.3]{RV}. Smooth truncations of these exhaustion functions provide the required separating functions.}

More specifically, we present classes of $(a,b,\Om,X,Y)$ satisfying the $(\k,\g)$-separation conditions. 
	Explicit constructions of  separation functions satisfying \eqref{kCond}--\eqref{gCond} for these examples are given in  Appendix \ref{secf}.
	\begin{itemize}
		\item 
		Let $\Om=\Rb^n$ and assume  \eqref{27} holds with $\al=1$.
		Then, for any $\delta>0$,
		there exist $\k=\k(n,\beta,\delta)$
		and $\g=\g(n)$ s.th.~any  disjoint non-empty   subsets $X,\,Y$ with $d_{XY}\ge\delta$  are $(\k,\g)$-separated.

		\item
		Let $(M,g)$ be a complete non-compact Riemannian manifold with bounded geometry, and let
		$\Om$ be a  smooth domain  in $M$. 
		Assume $	 \norm{a}_{L^\infty(\Om\times I)}\le 1,$ and $
\norm{\abs{b}+\abs{ \div a}}_{L^\infty({\Om\times I})}<\infty$.
		Then there exist $\k,\g>0$, s.th.~if $R>r>0$ and $R-r$ is sufficiently large, and $X,\,Y$ are  non-empty subsets $X,\,Y\subset\Om$ satisfying  
		$
		Y\subset\overline{B_r(o)}$ and $
		X\subset B_R(o)^\cp$ for some $o\in\Om$, with $d_{XY}=R-r$ and  $B_{R+2}(o)\subset\Om$, then $X,\,Y$ are
		$(\k,\g)$-separated.
		See \figref{figGeo} for an illustration. 	 
	\end{itemize}
	\begin{figure}[b]
		\centering
		\begin{tikzpicture}[scale=1.2]
			
			% Draw Omega boundary
			\draw[pink,scale=1.2, fill=pink]
			(-3.1,-1.2)
			.. controls (-3.3,1.6) and (-1.8,2.7) ..
			(1.8,2.5)
			.. controls (3.3,2.2) and (3.6,-1.8) ..
			(3.0,-2.1)
			.. controls (1.2,-2.5) and (-1.0,-2.3) ..
			(-2.5,-2.0)
			.. controls (-3.0,-1.8) and (-3.1,-1.5) ..
			cycle;
			\node at ( 2,2.3) {$\Omega$};
			
			\draw[white,scale=.9, thick,xshift=-.444cm,fill=white]
			(-2.1,-2.6)
			-- (-2.25,1.9)
			-- (-2.0,1.9)
			-- (1.7,2.4)
			-- (2.2,2.4)
			-- (2.5,-1.7)
			-- (2.1,-2.0)
			-- (1.0,-2.4)
			-- (1.8,-2.2)
			-- (2.4,-2.0)
			-- (2.05,-2.7)
			-- cycle;
			\node at (-3,-1.5) {$X$};
			
			% Draw the large ball B_{r+R}(y) as dashed circle
			\draw[thick, red!70,dashed] (0,0) circle (1.72);
			\node at (0,-2) {$ {B_{ R}(o)^\cp}$};
			
			% Draw the small ball B_r(y) as dashed circle
			\draw[thick, red!70,dashed] (0,0) circle (0.97);
			\node at (-1.2,0.3) {$\overline{B_r(o)}$};
			
			\draw[thick] (0,0) circle (0.01);
			\node at (-.1,-0.08) {$o$};
			
			% Draw blob region for Y inside B_r(y)
			\draw[green,fill=green!30, thick,xshift=.744cm] 
			(-1,0.1) .. controls (0,0.56) and (0.46,0.1) .. (0.14,-0.2)
			.. controls (-0.1,-0.4) and (-0.4,0.38) .. cycle;
			\node at (0.8,0) {$Y$};
			
			\draw[thick,<->,dashed] (1,0) -- (1.72,0);
			\node at (1.4,-.16) {$d_{XY}$};

		\end{tikzpicture}
		
		\caption{Schematic diagram for the geometric setup in the second example.
			Note that no regularity on the boundary of $X$ or $Y$ is  assumed.
		}
		\label{figGeo}
	\end{figure}

\begin{proof}[Proof of \thmref{thm} assuming \thmref{thm1}]
	
	 {First, we rescale time $t$ as $t\mapsto\al t$, and divide the coefficients $a,\,b,\,c$ in \eqref{Lgen} by $\al$ to arrive at the equation %
	\begin{align}
		\label{214a}
		\di_t u = L'u,\quad u :\Rb^n\times I_\al\to\Rb,
	\end{align} where $I_\al=\al I$ and, with $(a',b',c')=\frac1\al (a,b,c)$, 
\begin{align}
	\label{214b}
		 		L' u = \div (a'  \, \grad u ) + \inn{ b'}{\grad u}  +c' u .
\end{align} 
	Next, we apply \thmref{thm1} to \eqref{214a} with \eqref{214b}. 
By \propref{propB1}, since $(a',b')$ satisfy  conditions \eqref{27} with $\al=1$, any  disjoint non-empty  subsets $X,\,Y$ with $d_{XY}\ge\delta$  are $(\tilde \k,\g)$-separated. }  {Hence, it follows from \thmref{thm1}  in the Euclidean case with $(a,b,c,\k)\mapsto (a',b',c',\tilde\k)$ that
\begin{align}
			\norm{\chi_X P_t\chi_Y}_{L^p(\Rb^n)\to L^p(\Rb^n)} \le   \exp\del{-\frac{(\g d_{XY}-\k t)^2}{ 4\al  t}},\label{216}
\end{align}
 with $  \k:=\al\,\tilde \k\del{n,\frac{\beta}{\al},\delta}$ and 
for all $0<t< {\g d_{XY}}/{  \k}$. Finally, setting $\al'=\al/\g^2$ and $\k'=\k/\g$ in \eqref{216} gives \eqref{mainSp}. }\end{proof}

	 Next, assume $P_t$ has a classical integral kernel, i.e., $$P_t u(x)=\int p_t(x,y) u(y)\,dy$$ for some function $p_t\ge0$.
	 We prove:

	 \begin{corollary}[$L^1$-tail estimate for $p_t$]\label{cor:intbound}
			Assume   there exist $\k,\,\g>0$ s.th.~$ B_{\eps r}(x_0)$ and $B_{(1-\eps)r}(x_0)^\cp$ are $(\k,\g)$-separated for a.e.~$x_0\in\Om$, $r>0$, and $0<\eps<\frac13$. 
	 Then we have, for all   $0< t <  \g r /\k$ and a.e.~$x\in\Om$,
	 	\begin{align}
	 		\label{16}
	 		\int _{B_r(x)^\cp}{p_t(x,y)}\,dy\le  \exp\del{-\frac{(\g r-\k t)^2}{4t}}.
	 	\end{align} 
	 	
	 \end{corollary}
 
	 \begin{proof} 

	 			We follow the argument in \cite[Remark.~3.3]{GH}, with the geometric splitting illustrated in \figref{figDec}.  
	 			
	 		Take 	$X= B_{\eps r}(x_0)$ and $Y=B_{(1-\eps)r}(x_0)^\cp$, for some $x_0\in\Om$, $r>0$, and $0<\eps<\frac13$, such that $X,\,Y$ are $(\k,\g)$-separated. Then applying \eqref{122} with $p=\infty$ gives, for    $\k t<\g(1-2\eps)r$
	 		$$\sup_{x\in B_{\eps r}(x_0)} P_t\chi_{B_{(1-\eps)r}(x_0)^\cp}(x) \le \exp\del{-\frac{(\g(1-2\eps)r-\k t)^2}{4t}} .  $$
	 	For every  $x\in B_{\eps r}(x_0)$, we have $B_{r}(x)\supset B_{(1-\eps )r}(x_0)$ by the triangle inequality, and therefore
	 	$$	\int _{B_r(x)^\cp}{p_t(x,y)}\,dy\le 	\int _{B_{(1-\eps)r}(x_0)^\cp}{p_t(x,y)}\,dy=P_t\chi_{B_{(1-\eps)r}(x_0)^\cp}(x).$$
	 	Combining these  shows that  for a.e.~$x\in B_{\eps r}(x_0)$,
	 	$$\int _{B_r(x)^\cp}{p_t(x,y)}\,dy\le \exp\del{-\frac{(\g(1-2\eps)r-\k t)^2}{4t}}.$$
	 	Letting $x_0$ vary  in a countable dense subset of the admissible full-measure subset of $\Om$ yields the same inequality for a.e.~$x\in\Om$ and $\k t<\g(1-2\eps)r$. Finally, sending $\eps\downarrow0$ along a sequence gives \eqref{16}, as desired. 
	 \end{proof}

	\begin{figure}[h]

	\begin{tikzpicture}[scale=.8]

		% Draw the large ball B_{r+R}(y) as solid circle
		\draw[thick, red!70] (0,0) circle (4);
		\filldraw[thick] (0,0) circle (0.05);
		\node at (-0.18,0) {$x$};
		
		\draw[->] (0,0) -- (0,4) node[midway, left] {$r$};
		
		\draw[thick, red!70, dashed] (0.7,0) circle (1);
		\filldraw[thick] (0.7,0) circle (0.05);
		\node at (0.7,.18) {$x_0$};
		
		\draw[->] (0.7,0) -- (0.7,-3) node[midway, left] {$(1-\eps)r$};
		
		\draw[thick, dashed, red!70] (0.7,0) circle (3);

		% Draw arrow indicating distance R
		\draw[<->] (1.7,0) -- (0.7,0) node[midway, below] {$\eps r$};
	\end{tikzpicture}
	\caption{Schematic diagram for the geometric decomposition used to prove Corollary \ref{cor:intbound}. }
	\label{figDec}
\end{figure}

	\subsection{Nonlinear parabolic equations} \label{secExt}

	In this section, we show that our approach  encompasses a general class of nonlinear parabolic PDE, again allowing for  degeneracy in the strong sense that $a$ may vanish on an open set. 
	
	To fix ideas,  {we state the result in $L^1(\Rb^n)$}, in the time interval $I=(0,T)$ for some $0<T\le \infty$.
We consider the nonlinear degenerate parabolic equation  
	\begin{align}
		\label{nl-para-eq}		\begin{cases}
			\partial_t u = L(u) u, & \text{in }  \Rb^n\times I,\\
			u(\cdot, 0) = u_0, &
		\end{cases}
	\end{align} 
	where 	
	\begin{align}\label{Lu}
		L(u) =	  \div (a(u)  \, \grad (\cdot) ) + \inn{ b(u)}{\grad  (\cdot) }  +c(u)  ,
	\end{align}
and the coefficients $a(u)$, $b(u)$, and $c(u)$ are measurable functions of $(x,t)\in \Rb^n\times I$. 
		Given $u_0\in L^1(\Rb^n)$, we say $u$ is  an $L^1$-weak solution to \eqref{nl-para-eq}, if $u\in L^\infty(I,L^1(\Rb^n))$ and
\begin{align}
	\label{DSn}
	\int_I \int_{\Rb^n} u (\di_t+L(u)^*) \varphi\,dx\,dt+\int_{\Rb^n}u_0 \varphi(\cdot,0)\,dx =0, 
\end{align}
for all $ \varphi\in C_c^\infty(\Rb^n\times [0,T))$.

Given $\k,\,\g>0$ and a function $u:\Rb^n\times I\to\Rb$, we say $X,\,Y$ are $(\k,\g,u)$-separated  if there exists $\psi\in C_b^\infty(\Rb^n)$ s.th.~\eqref{gCond} holds and
\begin{align}
	&	\es_{\Rb^n\times I}\abs{\inn{a(u)\grad\psi}{\grad\psi}}\le1,\quad	 \es_{\Rb^n\times I}\del{\abs{\div(a(u)\grad\psi)}+\abs{\inn{b(u)}{\grad\psi}}}\le \label{kCondu}
	\k.
\end{align}
We prove the following:
	\begin{theorem}[Nonlinear diffusion bound in $\Rb^n$] \label{thm2}

		  {Let $u$ be a non-negative $L^1$-weak solution to \eqref{nl-para-eq}, with initial state $u_0\in L^1(\Rb^n)$, $u_0\ge0$, s.th.~$t\mapsto u(t)$ is weakly continuous in $L^1$.  Assume  
			the coefficients of $L(u)$ satisfy
			the following regularity conditions:
			\begin{align}
				\label{bCond}   
				\begin{cases}
					& a(u)\in L^\infty(I,W^{1,\infty} (\Rb^n)),\quad   \frac{b(u)}{1+\abs{x}}\in L^\infty(I, L^\infty (\Rb^n)),\\
					& \div b(u),\,c(u)\in L^\infty(I, L^\infty (\Rb^n)),\\
						&a (u)\ge 0, \quad  {\div b(u)  -c(u) \ge 0},\quad   c(u)  \le0,\quad \text{a.e.~in }\Rb^n\times I.
				\end{cases}
			\end{align} 
			} 
			Then, for any $\k,\,\g>0$ and disjoint non-empty measurable $(\k,\g,u)$-separated subsets $X,\,Y\subset \Rb^n$, we have, for $u_0$ supported in $Y$ and all $t\in I$ with $\k t <  \g d_{XY}$,
			\begin{align}
				\label{123}
				\norm{u(t)}_{L^1(X)} \le   \exp\del{-\frac{(\g d_{XY}-\k t)^2}{ 4 t}}\norm{u_0}_{L^1(Y)}. 
			\end{align}

	\end{theorem} 
	
	Theorem \ref{thm2} is proved in Section \ref{sec42}. 
	
		To obtain bounds for an admissible class of subsets independent of $u $, one can use a priori estimates on $a(u)$ and $b(u)$.  We illustrate this with an application to the McKean-Vlasov equation in \secref{sec:examples}.

		\section{Proof of \thmref{thm1}}\label{secPfThm1}

		The proof is organized as follows: In the Sects.~\ref{sec31}--\ref{sec22}, we prove \thmref{thm1} for $p=1$. 
		In \secref{sec23}, we conclude  \thmref{thm1}  by  a duality argument and interpolation. 
		
			\subsection{Key relation}\label{sec31}
			
			Let $\phi:\Om\to \Rb$  be a   signed  cutoff function adapted to the geometry of $X$ and $Y$, to be determined later. We introduce the exponential tilting multiplication operator
			\begin{equation}\label{Tdef}
				T: u(x)\mapsto e^{\phi(x)}u(x).
			\end{equation}
			Clearly, $T$ is invertible and gives 	exponential  weight adapted to $X$ and $Y$.
			
		Next, we write, for any $u\in L^1$, 
			\begin{align}
				\label{splitting}
				\chi_X P_t\chi_Yu= \chi_X T^{-1}     [(T P_tT^{-1})   T \chi_Y u],
			\end{align} which leads, through H\"older's inequality, to the rough bound
			\begin{align}\label{eq:a1}
				\norm{\chi_X P_t\chi_Y}_{L^1\to L^1}\le\norm{\chi_X T^{-1} }_{L^\infty} \norm{ T P_tT^{-1}}_{L^1\to L^1} \norm{ T \chi_Y }_{L^\infty}.
			\end{align}
		 {	In the next two subsections, we will estimate the terms on the r.h.s. The second term in the r.h.s.~leads to an exponential growth in $t$, 
			and the first and last geometrical   terms, to exponential decay in $d_{XY}$. }

			\subsection{Bound on deformed propagator}\label{sec32}

			 Let $L$ be as in \eqref{Lgen}.
			Given a twice differentiable function $\phi$, we define, with    $(\cdot)_+=\max(\cdot,0)$, \begin{align}
				\label{Adef}
				A_L[\phi](t):=\es_{x\in \Om}
				\del{\inn{a\nabla \phi}{\nabla \phi}+ 
				\div(a\nabla \phi)-\inn{b}{\grad \phi}}_+.
			\end{align}
			The main result in this section is the following   $L^1\to L^1$ estimate for the deformed propagator:
			\begin{proposition}\label{prop16}
							Let \eqref{aCond} hold. Let $\phi\in C_b^\infty(\bar\Om)$. Assume $t\mapsto A_L[\phi](t)$ is locally integrable, and $  
							  \norm{\inn{a\grad\phi}{\grad\phi}}_{L^\infty(\Om\times (0,\infty))}<\infty$.
							Then, for all non-negative $v\in L^1(\Om)$   and  {$t>0$,} 
				\begin{align}
					\label{estGen}
					&\int_\Om e^\phi P_t e^{-\phi} v  
					\le  \exp\del{\int_0^t A_L[\phi](\tau)\,d\tau} \int_\Om v.
				\end{align}
	
			\end{proposition}
\begin{proof}
				Let
\begin{align}
	\label{wDef}
			\tilde u(t)
			:=P_te^{-\phi}v,
			\quad
			w(t):=e^\phi \tilde u(t),
			\quad
			M(t):=\int_\Omega w(t)\,dx.
\end{align}
		We claim	$M'(t)\leq A_L[\phi](t)M(t)$ in distributional sense, i.e.,
			\begin{equation}
				\label{Mest}
				-\int_0^\infty M(t)\theta'(t)\,dt
				\leq M(0)\theta(0)+
				\int_0^\infty A_L[\phi](t)M(t)\theta(t)\,dt,
			\end{equation}
			for every non-negative \(\theta\in C_c^\infty([0,\infty))\).
			
 {Assuming \eqref{Mest} holds, then it follows from  the distributional Gr\"onwall's inequality that \eqref{estGen} holds for a.e.~$t>0$. The continuity result \ref{P1} of \propref{prop1} then gives \eqref{estGen} for all $t>0$. }

			Thus it remains to prove \eqref{Mest}. 
By \ref{P3} of \propref{prop1},  $\tilde u$ is a $L^1$-weak solution to \eqref{PE1}--\eqref{BC} with initial condition $\tilde u_0=e^{-\phi}v$. Explicitly, we have
			\begin{equation}
				\label{eq:u-epsilon-weak}
				\int_0^\infty\!\!\int_\Omega
				\tilde u(\partial_t+L^*)\varphi\,dx\,dt
				+\int_\Omega e^{-\phi}v\,\varphi(\cdot,0) \,dx=0,
			\end{equation}
		for all $ \varphi\in C_c^\infty(\overline \Om\times [0,\infty))$ with $\inn{a\grad \varphi }{\nu}\vert_{\di\Om}=0$ if $\di\Om\ne\emptyset$, where
			\[
			L^*\varphi
			=\operatorname{div}(a\nabla\varphi)
			-\operatorname{div}(b\varphi)+c\varphi.
			\]

			Consider first the Riemannian case, where \(\bar \Omega\) is a compact submanifold with a smooth boundary.
			Choose
			$
			\varphi(x,t)=\theta(t)e^{\phi(x)}.
			$
			The assumption
			\(\nabla\phi=0\) on \(\partial\Omega\) makes $\varphi$ an admissible test
			function. Hence, testing \eqref{eq:u-epsilon-weak} with this $\varphi$ and using  definition \eqref{wDef} yields  
			\begin{align}
				\label{Meq}
				-\int_0^\infty M(t)\theta'(t)\,dt
				=M(0)\theta(0)+\int_0^\infty\theta(t)\int_\Omega
				w(t)e^{-\phi}L^*e^\phi\,dx\,dt.
			\end{align}
			We compute, for smooth $\phi$, 
			\begin{equation}
\notag
				e^{-\phi}L^*e^\phi
				=Q+c-\operatorname{div}b,
			\end{equation}
			where 
						\begin{equation}
\notag
							Q
								:=\langle a\nabla\phi,\nabla\phi\rangle
								+\operatorname{div}(a\nabla\phi)
								-\langle b,\nabla\phi\rangle.
							\end{equation}
			By definition \eqref{Adef}, we have $A_L[\phi](t)=\es_{x\in \Om}(Q)_+$. This, together with the assumption  \(\operatorname{div}b-c\geq0\) in \eqref{aCond}, implies
			\begin{align}
				Q+c-\operatorname{div}b
				\leq A_L[\phi],\quad \text{a.e.~in }\Om\times (0,\infty).
\notag
			\end{align}
		 {By the positivity-preserving property  \ref{P2} of \propref{prop1}, we have  \(w\geq0\). Hence, }combining the above  yields, for a.e.~$t>0$,
			\begin{align}
				\int_\Omega
				w(t) e^{-\phi}L^*e^\phi\,dx
				&\leq A_L[\phi](t)M(t).
				\label{eq:Q-integrated-bound}
			\end{align}
			Lastly, since \(\theta\geq0\), plugging \eqref{eq:Q-integrated-bound} back to  
			\eqref{Meq} 
			 gives
			\eqref{Mest}, as desired.

		Next, consider the Euclidean case  \(\Omega=\mathbb R^n\). Then  the function \(e^\phi\) is  no longer compactly supported, and thus $\varphi=\theta e^\phi$ is not an admissible test function. To overcome this, we  choose a sequence of cutoff functions \(\chi_R\in C_c^\infty(\mathbb R^n)\), $R\ge1$, satisfying
\begin{align}
		\label{chiDef}	\begin{cases}
				&0\leq\chi_R\leq1,
			\qquad
			\chi_R=1\text{ in }B_R,
			\qquad
			\chi_R=0\text{ in }B_{2R}^\cp,\\
			&	|\nabla\chi_R|\leq\frac{C}{R},
			\qquad
			|\grad ^2\chi_R|\leq\frac{C}{R^2}.
		\end{cases}
\end{align}
			Put
			\[
			M_R(t):=\int_{\mathbb R^n}\chi_Rw(t)\,dx.
			\]
We compute
			\begin{align}
				e^{-\phi}L^*(e^\phi\chi_R)
				={}&\chi_R
				\bigl(Q
				+c-\operatorname{div}b\bigr)+D_{R},\notag\\
				D_{R}:=&\operatorname{div}(a\nabla\chi_R)
				+\langle2a\nabla\phi-b,
				\nabla\chi_R\rangle.\notag
			\end{align}
		Since  $0\le \chi_R\le 1$, we have $$e^{-\phi}L^*(e^\phi\chi_R)\le A_L[\phi] +D_{R},\quad \text{a.e.~in }\Om\times (0,\infty).$$
		Now choose $
		\varphi(x,t)=\theta(t)e^{\phi(x)}\chi_R(x)$, which is an admissible test function. 
			 {Since $w\ge0$,} testing \eqref{eq:u-epsilon-weak} with this $\varphi$  gives
			\begin{align}
				\label{eq:M-R-identity}
				&-\int_0^\infty M_R(t)\theta'(t)\,dt-M_R(0)\theta(0)\notag\\
				={}&\int_0^\infty\theta(t)\int_\Omega
				w(t)	e^{-\phi}L^*(e^\phi\chi_R)\,dx\,dt.\notag\\
				\le& \int_0^\infty \theta(t) \del{A_L[\phi](t) M(t)+ E_{R}(t)}\,dt,
			\end{align}
			where  $
				E_{R}(t)
				:=\int_{\mathbb R^n}w
				D_{R}dx. $  
						
						We now derive \eqref{Mest} from \eqref{eq:M-R-identity}.
 First, we note $ M_R(0)\to M(0)$ as $R\to\infty$.  We claim \begin{enumerate}[label=(\roman*)]
 	\item $E_R(t)\to 0$ for a.e.~$t$ as $R\to\infty$; \item $\abs{\theta (t) E_R(t)}$ is  bounded by some function in  $L^1(0,\infty)$ for a.e.~$t$ and all $R\ge1$. 
 \end{enumerate}Once these claims are proved, then it follows from  the  dominated convergence theorem that, as $R\to\infty$, 
				\[
				\int_0^\infty|\theta(t)E_{R}(t)|\,dt\to0 .
				\]
				Sending \(R\to\infty\) in
				\eqref{eq:M-R-identity} gives \eqref{Mest}, as desired.

It remains to prove the claims above. Using \eqref{chiDef} and H\"older's inequality, we find  for a.e.~$t>0$ that
			\begin{align}
				|E_{R}(t)|
				\leq C\bigg(&
				\frac{\|\nabla a(t)\|_\infty}{R}
				+\frac{\|a(t)\|_\infty}{R^2}
				+\frac{\|a(t)\|_\infty^{1/2}
					\|\inn{a\grad\phi}{\grad\phi}\|_\infty^{1/2}}{R}\notag\\
				&\hspace{2.5cm}
				+\frac{\norm{b(t)}_\infty}{ R}
				\bigg)M(t).
\notag
			\end{align}
This, together with the boundedness of coefficients in \eqref{aCond} and the assumption $	\|\inn{a\grad\phi}{\grad\phi}\|_{L^\infty(\Om\times (0,\infty))}<\infty$, implies $\lim_R E_{R}(t)=0$ for a.e.~$t$. This proves (i).
Next,  recall $\phi$ is bounded, $v\in L^1$,  and $P_t$ is contractive on $L^1$ by \ref{P2}. Hence,  we have  $M(t)\le C \norm{v}_{L^1}$ for all $t>0$. Since, moreover, $\theta$ is bounded, we have $\abs{\theta (t) E_R(t)}\le C \norm{v}_{L^1} \1_{\supp \theta}(t)$ for a.e.~$t>0$ and some $C$ depending only on $\phi,\,\theta,$ and the coefficients of $L$. Since $\supp\theta$ is compact, this provides the desired dominating function, and so (ii) is proved.
\end{proof}

			\begin{corollary}\label{cor33}
				Let \eqref{aCond} hold, and
				let $\phi$ and  $A_L[\phi]$ be as in \propref{prop16}. Then, for  all $t>0$, 
\begin{align}\label{TPT}
						\norm{ e^\phi P_t e^{-\phi}}_{L^1\to L^1} \le \exp\del{\int_0^t A_L[\phi](\tau)\,d\tau} .
\end{align}
			\end{corollary}

			\begin{proof}
		Take any $v\in L^1(\Om)$. Decomposing it as $v = v_+ - v_-$ with $v_\pm \ge0$, applying \eqref{estGen} respectively to $v_+$ and $v_-$, and using that   $P_t$ is positivity-preserving, we conclude 
							\eqref{TPT}. 
			\end{proof}

		\subsection{Proof of \thmref{thm1} for $p=1$}\label{sec22}

Step 1. 
Let $X,\,Y$ be $(\k,\g)$-separated subsets, and 
	let $\phi$ and  $A_L[\phi]$ be as in \propref{prop16}.
For $T$ given by \eqref{Tdef} and $X,\,Y\subset\Om$, we have
$$	\chi_X T^{-1} \le   e^{-\inf_X\phi}\chi_X,\quad
T \chi_Y\le    e^{\sup_Y\phi}\chi_Y.$$
These, together with \eqref{eq:a1}, \eqref{TPT}, and the definition of $\phi(X,Y)$ in \eqref{psiXY}, imply, for all $t>0$, 
\begin{align} \label{318}
	\norm{\chi_X P_t\chi_Y}_{L^1\to L^1}\le \exp\del{-\phi(X,Y)+\int_0^t A_L[\phi](\tau)\,d\tau} .
\end{align}

Step 2. 
Now  let $\psi\in C_b^\infty(\bar\Om)$ be a function satisfying \eqref{kCond}, and choose $\phi=\mu\psi$ for some $\mu>0$ to be determined. Then, owing to \eqref{kCond}, the functional $A_L$ from \eqref{Adef} is bounded for a.e.~$t>0$  as, \begin{align}\label{Aest}
A_L[\mu\psi](t)=&\es_{x\in \Om}
\sbr{\mu^2\inn{a\nabla \psi}{\nabla \psi}+ 
	\mu \del{\div(a\nabla \psi)-\inn{b}{\grad \psi}}}_+\notag\\
	\le& \mu^2  +  \mu \k.
\end{align}
Combining \eqref{318}--\eqref{Aest} and using the definition \eqref{psiXY} gives
\begin{align}\label{321b}
	\norm{\chi_X P_t\chi_Y}_{L^1\to L^1}\le e^{-G_1(\psi,\mu,t)},
\end{align}
where 
$$G_1(\psi,\mu,t)=-\mu^2 t + \mu  ({\psi(X,Y)-\k t}).$$
For $\psi$ satisfying the lower bound \eqref{gCond}, estimate \eqref{321b} gives  
\begin{align}
\notag
G_1(\psi,\mu,t)\ge 	G_2(\mu,t):=-\mu^2 t + \mu  ({\g d_{XY}-\k t}).
\end{align}
For fixed $t$, the function $G_2(\mu,t)$ is monotone decreasing in $\mu$ unless
$$\g d_{XY} > \k t.$$
In this ballistic time interval,  maximizing $G_2(\mu,t)$ over $\mu$ yields 
\begin{align}
	G_*(t):=\sup_{\mu>0} G_2(\mu,t) =  	\frac{({\g d_{XY}-\k t})^2}{4 \,t},\notag
\end{align}
attained at $\mu_*=\frac{\g d_{XY}-\kappa t}{2t}$. Since the l.h.s.~of \eqref{321b} is independent of $\mu$ and $\psi$, this  gives
the  bound \eqref{122}  for $p=1$. 
\qed

		\subsection{Completing the proof of \thmref{thm1} }\label{sec23}

		Estimate \eqref{122} has already been proved for $p=1$.  We next prove the
		endpoint $p=\infty$ by applying the $p=1$ estimate to the time-reversed
		adjoint equation, and then interpolate.

		Fix $t>0$. Applying the $L^1$--$L^\infty$ duality, we find 
		\begin{align}
			\label{dual}
			\|\chi_XP_t\chi_Y\|_{\infty\to\infty}=\|\chi_YP_t^*\chi_X\|_{1\to1}.
		\end{align}
		By \lemref{lem:adjoint-compatible-viscosity}, $P_t^* = Q_t$ on $L^1$, where $Q_r,\,0\le r\le t$ is a propagator generated by the operator
	\begin{align}\label{LsDef}
			L^\sharp(r)v
		:=L(t-r)^* v =\operatorname{div}(a^\sharp(r)\nabla v)
		+\langle b^\sharp(r),\nabla v\rangle+c^\sharp(r)v,
	\end{align}
		where
		\[
		a^\sharp(r)=a(t-r),\quad
		b^\sharp(r)=-b(t-r),\quad
		c^\sharp(r)=c(t-r)-\operatorname{div}b(t-r).
		\]
		If $L$ satisfies \eqref{aCond}, then so does $L^\sharp$. Furthermore, if  $\psi$ is a separation function satisfying \eqref{kCond}--\eqref{gCond}, then  
		\begin{align}
			&	A_{L^\sharp}[-\mu\psi](r)\notag\\
				=&\es_{x\in \Om}
				\sbr{\mu^2\inn{a(t-r)\nabla \psi}{\nabla \psi}- 
					\mu \del{\div(a(t-r)\nabla \psi)+\inn{b(t-r)}{\grad \psi}}}_+\notag\\
				\le& \mu^2  +  \mu \k,\notag\\
		&	(-\psi)(Y,X)=	\psi(X,Y)\ge \g d_{XY}.\notag
		\end{align}
Hence, applying \corref{cor33} to $Q_t$ with $\phi=-\mu\psi$,
we find, similar to \eqref{321b}, 
	\begin{align}\notag
		 \|\chi_YQ_t\chi_X\|_{1\to1} 
		\le& e^{-G_1(\psi,\mu,t)}. 
	\end{align}
Recalling the relations \eqref{dual} and   $Q_t=P_t^*$, we conclude 
	\[
	\|\chi_XP_{t}\chi_Y\|_{\infty\to\infty}
	\le e^{-G_1(\psi,\mu,t)}.
	\]
	From here we proceed exactly as Step 2 in \secref{sec22} to conclude 
 that \eqref{122} holds for $p=\infty$.
 
  Finally, Riesz--Thorin interpolation between
		$p=1$ and $p=\infty$ gives \eqref{122} for all $1<p<\infty$. This completes the proof of \thmref{thm1}. \qed

		\section{Proof of \thmref{thm2}}\label{sec42}
		The proof follows the same line as Sects.~\ref{sec32}--\ref{sec22}, so we only sketch the key steps.

		Let $u$ be a weak solution  as in \thmref{thm2}, and take $\phi\in C_b^\infty(\Rb^n)$ to be determined later.  We set 
		\begin{align}\notag
	w(t):=e^\phi u(t),
	\quad
	M(t):=\int_{\Rb^n} w(t)\,dx, 		\quad 	M_R(t):=\int_{\mathbb R^n}\chi_Rw(t)\,dx.
\end{align}
Here $\chi_R$ is a family of cut-off functions satisfying \eqref{chiDef}. 
	Using definition \eqref{DSn} and the non-negativity of $u$, we find, 
					for every non-negative \(\theta\in C_c^\infty([0,T))\), 
								\begin{align}
					&	-\int_0^T M_R(t)\theta'(t)\,dt 
						\notag\\
						\le& M_R(0)\theta(0) +\int_0^T \theta(t) \del{A_{L(u)}[\phi](t) M(t)+ E_{R,u}(t)}\,dt,\label{41}
					\end{align}
					where   $A_{L(u)}$ is as in \eqref{Adef}, and
					\begin{align}\notag
							E_{R,u}(t)
						:=&\int_{\mathbb R^n}w
						D_{R,u}dx,\\
						D_{R,u}(x,t):=&\operatorname{div}(a(u)\nabla\chi_R)
						+\langle2a(u)\nabla\phi-b(u),
						\nabla\chi_R\rangle.
					\end{align}
					
			Consider the first term in the r.h.s.~of \eqref{41}. 		Since $\supp u_0\subset Y$ and $u_0\ge0$, H\"older's inequality gives
					\begin{align}\notag
						M_R(0)\le \norm{u_0}_{L^1(Y)}e^{\sup_Y\phi}.
					\end{align}
			Next, for a.e.~$t\in I$, since the derivatives of \(\chi_R\) are supported in
		\(\mathcal A_R:=\{R\leq|x|\leq2R\}\), we have by H\"older's inequality that 
		\begin{align}
			|E_{R,u}(t)|
			\leq &C	\int_{\mathcal A_R}w(t)\,dx\,\bigg(
			\frac{\|\nabla a (u)\|_{L^\infty(\Rb^n)}}{R}
			+\frac{\|a(u) \|_{L^\infty(\Rb^n)}}{R^2}
			\notag\\
			&+\frac{\|a(u) \|_{L^\infty(\Rb^n)}^{1/2}	\|\inn{a(u)\grad\phi}{\grad\phi}\|_{L^\infty(\Rb^n)}^{1/2}}{R}
			+ \norm{\frac{b(u)}{1+\abs{x}}}_{L^\infty(\Rb^n)}\bigg).
			\label{eq:E-R-bound}
		\end{align}
		 {Since $w\ge0$ and \(w(t)\in L^1(\mathbb R^n)\), we have  $\int_{\mathcal A_R}w(t,x)\,dx\to 0$ as $R\to\infty$.} This, together with the coefficients assumption in \eqref{bCond} and the first condition in \eqref{kCondu}, implies $\lim_R E_{R,u}(t)=0$ for a.e.~$t$.
Sending $R\to\infty$ and using Gr\"onwall's inequality and dominated convergence theorem, we conclude
		\begin{align}\notag
							&M(t)
			\le   \norm{u_0}_{L^1(Y)}\exp\del{\sup_Y \phi+\int_0^t A_{L(u)}[\phi](\tau)\,d\tau}  ,
		\end{align}
		first for a.e.~$t\in I$, and then extended to all $t\in I$ by the weak continuity of $t\mapsto u(t)$.

		Now write  $\chi_X u  = \chi_X e^{-\phi}w $. Then applying H\"older's inequality again gives, for any $t\in I$, 
\begin{align}\notag
		\norm{u(t)}_{L^1(X)}=&	\norm{\chi _X u(t)}_{L^1(\Rb^n)} \notag\\
		\le&\norm{\chi_X e^{-\phi} }_{L^\infty(\Rb^n)} M(t)\notag\\
		\le&  \norm{u_0}_{L^1(Y)} \exp\del{-\phi(X,Y)+\int_0^t A_{L(u)}[\phi](\tau)\,d\tau} .\label{44}
\end{align}
Now choose $\phi=\mu\psi$, with $\psi$ satisfying \eqref{kCondu}, \eqref{gCond}.
Owing to \eqref{kCondu}, we have
$A_{L(u)}[\phi]\le \mu^2+\mu\k$; cf.~\eqref{Aest}. 
Hence, using \eqref{44}  in place of \eqref{318}, and 
	proceeding  exactly as in Step 2 in \secref{sec22}, we arrive at the desired result. 
\qed

		\section{Application to the McKean-Vlasov equation}
		\label{sec:examples}

	The McKean-Vlasov equation reads
		\begin{align}
			\label{MV}
			\partial_t f+v \cdot \nabla_x f+(K \star \rho) \cdot \nabla_v f=\sigma \Delta_v f.
		\end{align}
		Here $\si>0$ is the diffusion strength, $f: \Rb^n_x \times \Rb^n_v\times (0,\infty)\to \Rb$ is the phase-space density, $K:\Rb^n_x\to\Rb^n_x$ is the force   (vector field),   $\rho(x,t)= \int _{\Rb^n}fdv$ is the spatial  density, and  $(K \star \rho)_i = K_i*\rho$. 
		
	Eq.~\eqref{MV} is the Fokker-Planck (forward Kolmogorov) equation for the McKean-Vlasov nonlinear SDE, which is the mean-field limit of a system of interacting particles with a velocity noise. 

	In the noninteracting case, i.e., when $K\equiv 0$, eq.~\eqref{MV} admits a fundamental solution, given explicitly by Kolmogorov's formula 
(\cite{Kol}) as
	\begin{align}
		\label{eq:Kolmogorov-kernel}
		\Gamma_t(x,v;x_0,v_0)
		=\left(\frac{\sqrt3}{2\pi\sigma t^2}\right)^n
		\exp\!\left(
		-\frac{|v-v_0|^2}{4\sigma t}
		-\frac{3\left|x-x_0-\frac t2(v+v_0)\right|^2}
		{\sigma t^3}
		\right).
	\end{align}
	Thus the stochastic growth scale is $t^{1/2}$ in velocity and $t^{3/2}$ in
	position. 
	
	 In what follows, we recover the velocity scaling  with general $K$. 
	More precisely, 	we study the propagation properties of $f$ in the velocity space, i.e., transport between subsets  of the form $$\widetilde S = \Rb^n_x\times S,\quad S\subset \Rb^n_v.$$ 
We prove 
		\begin{corollary}[Diffusion bounds for the McKean-Vlasov equation]

	Let $f$ be a non-negative $L^1$-weak solution to \eqref{MV}  with $f_0\in L^1(\Rb^n_x\times \Rb_v^n)$, $f_0\ge0$, s.th.~$t\mapsto f(t)$ is weakly continuous in $L^1$. 
					Assume $\si=1$, $K\in W^{1,\infty}(\Rb^n)$.
				Then, for any $ \beta>\norm{K}_{L^\infty}\norm{f_0}_{L^1(\Rb^n_x\times \Rb^n_v)}$ and  non-empty convex   subsets $X,\,Y\subset \Rb^n_v$ with $d_{XY}>0$, we have, for $f_0$ supported in $\widetilde Y$ and all $t>0$, 
			\begin{align} 
				\| f(t)\|_{L^1(\widetilde X)}\le  \exp\del{-\frac{(d_{XY}-\beta t )^2}{4  t}}\| f_0\|_{L^1(\widetilde Y )}, \quad 	\beta t<  {d_{XY}} .
			\end{align} 
		\end{corollary}

\begin{proof}

To begin with,
we arrange the McKean-Vlasov equation to the form $\di_t f = L(f)f$,   where  $L(f)$ is as in \eqref{Lu}, with the coefficients given by  \begin{align}\label{Kc}
		a\equiv \begin{pmatrix}
			0&0\\0& \1
		\end{pmatrix},\quad b(f) = (-v,-K \star \rho)=:(b_x,b_v(f)),\quad c\equiv 0.
	\end{align}
Below we derive a priori estimate on $b(f)$. 
We note first that $\div_{x,v} b \equiv 0$. Next, write $M:=\norm{f_0}_{L^1(\Rb^n_x\times \Rb^n_v)}$. Since $f$ is a non-negative weak solution to \eqref{MV}, for all $t>0$,  we have  $\norm{\rho(t)}_1=\norm{f(t)}_{L^1(\Rb^n_x\times \Rb^n_v)}\le  M$, and therefore, by Young's inequality,  $\norm{b_v(f)(\cdot,t)}_\infty\le \norm{K}_\infty \norm{\rho(t)}_1\le \norm{K}_\infty M$. Similarly, we have $\norm{\grad _{x,v}b }_{L^\infty(\Rb^n_x\times \Rb^n_v)}\le C(1+\norm{\grad K}_{\infty}M) $.
Hence the coefficients of $L(f)$ satisfy the assumptions of \thmref{thm2}. 

We now claim $\widetilde X,\,\widetilde Y$ are $(\beta,1,f)$ separated. Once this claim is proved, the desired result follows by
applying \thmref{thm2} and using the fact that $d_{\widetilde X\widetilde Y}=d_{XY}$. 

Let  $d=d_{XY}>0$. 
Since $X,\,Y$ are convex, by the separation hyperplane theorem, there exists a unit vector $e\in\Rb^n$ s.th.~if $
	s_Y:=\sup_{y\in Y}\inn e y,$ $
	s_X:=\inf_{x\in X}\inn e x,$
	then $s_X,s_Y$ are finite and $s_X-s_Y=d$.
Now let $
	\eps:=\frac12\left(
	\beta-\norm K_\infty M
	\right)>0$, and 
	choose $F\in C_b^\infty(\Rb)$ such that
	\[
	0\le F'\le1,\qquad
	F'=1\ \text{on }[s_Y,s_X],
	\qquad
	\norm{F''}_\infty\le\eps.
	\] Define $
	\psi(x,v):=F\bigl(\inn e v\bigr).$
	Then $\psi\in C_b^\infty(\Rb^{2n})$ and
	\[
	\inn{a\grad\psi}{\grad\psi}
	=\abs{F'(\inn e v)}^2\le1.
	\]
	Furthermore, we compute
	\[
	\div(a\grad\psi)=\Delta_v\psi=F''(\inn e v),
	\quad
	\inn{b(f)}{\grad\psi}
	=-F'(\inn e v)\inn{K\star\rho}e.
	\]
	Hence
	\[
	\es_{\Rb^{2n}\times[0,\infty)}
	\left(
	\abs{\div(a\grad\psi)}
	+\abs{\inn{b(f)}{\grad\psi}}
	\right)
	\le\eps+\norm K_\infty M<\beta.
	\]
	This verifies \eqref{kCondu} with $\k=\beta$.
	Moreover, by monotonicity of $F$,
	\[
	\psi(\widetilde X,\widetilde Y)
	\ge F(s_X)-F(s_Y)
	=s_X-s_Y=d.
	\]
This verifies \eqref{gCond} with $\g=1$. Thus the claim is proved.  
	\end{proof}

		\section*{Acknowledgments}  
		The research of M.L.\ is supported by the DFG through the grants TRR 352 – Project-ID 470903074 and FOR 5413 – Project-ID 465199066, as well as by the European Union (ERC Starting Grant MathQuantProp, Grant Agreement 101163620).\footnote{Views and opinions expressed are however those of the authors only and do not necessarily
	reflect those of the European Union or the European Research Council Executive Agency. Neither
	the European Union nor the granting authority can be held responsible for them.} 
I.M.S. is supported by NSERC Grant NA7901. 
J.Z.~is supported by National Natural Science Foundation of China Grant 12401602, China Postdoctoral Science Foundation Grant 2024T170453, National Key R \& D Program of China Grant 2022YFA100740,   and the Shuimu Scholar program of Tsinghua University. He thanks  J.~Hu and B.~Zhu for helpful discussions. 
		
				\section*{Declarations}
		\begin{itemize}
			\item Conflict of interest: The Authors have no conflicts of interest to declare that are relevant to the content of this article.
			\item Data availability: Data sharing is not applicable to this article as no datasets were generated or analysed during the current study.
		\end{itemize}
		
		\appendix

\section{Existence of propagator} 
\label{secWP}

In this appendix, 	 {we construct $P_t$ by approximating $L$ with $L_\eps = L+\eps \Lap$,  using standard parabolic and semi-group theory for $L_\eps$, and extracting a limit of the propagators generated by $L_\eps$. 
A related construction is used by ter Elst et al (\cite{tERSZ, tERSZa}) to study degenerate elliptic operators using Dirichlet form.
	
In general, $P_t$ depends on the choice of approximating sequence $\eps_j$, and therefore may not be  unique.
	Uniqueness is a separate problem and typically requires higher regularity on the coefficients; see Le Bris-Lions \cite{LBL} for the uniqueness theory for Fokker-Planck type equations in the Euclidean case, which encompasses \eqref{Lgen}. 	}
	
	Throughout the appendices, set $I:=(0,\infty)$.
Fix $\eps>0$.  We define $
a_\eps := a + \eps \1$ and  
\begin{equation}\label{eq:approx}
	L_\eps  = \operatorname{div}(a_\eps \nabla (\cdot )) +\inn{ b  }{\nabla (\cdot)}  + c.
\end{equation}
Given $1\le p \le \infty$ and $u_0\in L^p(\Om)$, we consider the  Cauchy problem
\begin{align}
	\label{PE2}
	\begin{cases}
		\partial_t u^\eps = L_\eps u^\eps, & \text{in } \Omega\times I,\\
		u^\eps(\cdot, 0) = u_0,  &
	\end{cases}
\end{align}
subject to the Neumann 
boundary condition
\begin{align}
	\label{BC'}
	\inn{a_\eps\grad u^\eps}{\nu}=0\quad \text{if $\di\Om\ne\emptyset$}. 
\end{align}  
Since $a\ge0$ and so $a_\eps \ge \eps \1$, the equation $\di_t u^\eps = L_\eps u^\eps$ is uniformly parabolic. 
The next proposition  establishes the existence of solution for \eqref{PE2}--\eqref{BC'}. 
\begin{proposition}\label{prop:eps-prop}Let \eqref{aCond} hold.
	Then, for every $\eps>0$, there exists a 
	unique family of operators  $P_t^\eps\in \cB(L^2(\Om))$, $t\ge 0$, s.th.
	 \begin{enumerate}[label=(\roman*)]
		\item   {$t\mapsto P_t^\eps$ is strongly continuous  in  $\cB(L^2)$;}
		\item for every $t\ge 0$, $P_t^\eps$ is positivity-preserving and extends to a contraction  on $L^p(\Om)$,     $1\le p\le \infty$;  
		\item  for every $u_0\in L^2(\Om)$ and $t>0$,
		$u^\eps(t)=P_t^\eps u_0$ satisfies $u^\eps (0)=u_0$, $\eps \norm{\grad u^\eps(t)}_{L^2(\Om\times I)}^2\le \frac12\norm{u_0}_{L^2(\Om)}^2$, and, in the distributional sense on $\cD'(I)$,
		\begin{align}
			\label{B5}
			\di_t \int_\Om u^\eps v \,dx+\int_\Om\del{\inn{a_\eps\grad u^\eps}{\grad v} 
				- \inn{b}{\grad u^\eps}v 
				-  cu^\eps v }\,dx=0,
		\end{align}
		for all $ v\in H^1(\Om).$
	\end{enumerate}

\end{proposition}

\begin{proof}

This proposition follows from standard parabolic theory, so we will provide the references and  sketch the proof. 

Let $u_0\in L^2(\Om)$. Then it follows from
\cite[Chapter~XVIII, \S 3, Thms.~1--2]{DLb} that \eqref{PE2}--\eqref{BC'} has a unique  energy solution  $u^\eps$,  in the sense of \cite[Chapter~XVIII, eqs.~(3.21)--(3.23)]{DLb}. 
We first define $P_t^\eps$, $t\ge 0,$ on $L^2(\Om)$ by setting $P_t^\eps u_0 := u^\eps(t)$. Then (i) and (iii) hold by the energy identities of $u^\eps$. Furthermore, $P_t^\eps$ extends to a positivity-preserving contraction on $L^p(\Om)$ for all $1\le p\le \infty$.	In the case  $b,\,c=0$, this follows directly from the argument in	\cite[Chapter~XVIII, \S 4,  Thms.~3--4]{DLb}, and the general case follows by using the dissipativity condition in  \eqref{aCond}. We omit the details here.
\end{proof}

We now prove \propref{prop1} by using the construction above.

\begin{proposition}[Existence of propagator]
	\label{prop:visc-family}\label{propA}
 {Let $P_t^\eps,\, \eps>0,\,t\ge 0$ be as in Proposition~\ref{prop:eps-prop}. }
Then there exist 
a sequence $\eps_j\downarrow0$ and a family of operators $P_t\in \cB(L^2(\Om)), t\ge 0,$  s.th.~for every $T>0$ and $f,g\in L^2(\Om)$,
\begin{equation}\label{eq:luwot}
	\lim_{j\to\infty}	\sup_{ 0\le  t\le T}
	\left|\inn{(P_t^{\eps_j}-P_t)f}{g}_{L^2}\right|
	=0.
\end{equation}
Furthermore, $P_t$ satisfies \ref{P1}--\ref{P3}.  
\end{proposition}

\begin{proof}

Take any sequence $\eps\downarrow0$.  Let
	$\cD_0\subset C_c^\infty(\Om)$ be a countable dense rational-linear subspace in
	$L^2(\Om)$.  Given $f,g\in\cD_0$ and   $m\in\mathbb{N}$, set
	\[
	F_\eps(t)\equiv F_\eps(t,f,g):=\inn{P_t^\eps f}{g}_{L^2},\quad t\in[0,m].
	\]
By \eqref{B5}, we have, in the distributional sense on $\cD'(I)$,
	\begin{align*}
		\partial_tF_\eps(t)
		&=\inn{P_t^\eps f}{L_\eps(t)^*g}_{L^2}.
	\end{align*}
This, together with the $L^2$-contractivity of $P_t^\eps$, implies, for $t'<t$ and $0<\eps \le 1$, 
	\begin{align}\notag
				\abs{F_\eps(t)}&\le \norm{f}_2 \norm{g}_{2}\\
		|F_\eps(t)-F_\eps(t')|
		&\le\notag \norm f_2\int_{t'}^t\norm{L_\eps(s)^*g}_2\,ds,\\
		&\le C\abs{t-t'},\notag
	\end{align}
	where $C$ depends only on $\norm{f}_{L^2}, \norm{g}_{H^2}$, and the coefficients of $L$.  
	Therefore, for any fixed triple $(f,g,m)\in\cD_0\times \cD_0\times \mathbb N$, the functions $F_\eps(t), 0<\eps\le 1,$ are uniformly bounded and
	equicontinuous. Using the Arzel\`a--Ascoli theorem, together with a
	diagonal extraction over $(f,g,m)$, we obtain a 
	sequence $\eps_j$ and a family of continuous  functions, $F(\cdot,f,g)$, s.th.~for any $(f,g,m)\in\cD_0\times \cD_0\times \mathbb N$,
	\begin{align}\notag
		\lim_{j\to\infty}\sup_{ 0\le  t\le m}\abs{F_{\eps_j} (t,f,g)-F(t,f,g)}=0. 
	\end{align}
	Define $P_t,\,t\ge 0,$ first on $\cD_0$ by the bilinear form $\inn{P_tf}{g}=F(t,f,g)$. Then, for each $t$,  extend $P_t$ to all of $L^2$ by continuity. Since $F(\cdot,f,g)$ is continuous, it follows from a $3$-$\eps$ argument that $t\mapsto P_t$ is continuous in the w.o.t.~of $\cB(L^2)$, and \eqref{eq:luwot} holds.

The  positivity-preserving property is closed under the w.o.t.~convergence, and the construction above shows $P_t$ is a contraction on $L^2$. To see that $P_t$ extends to a contraction on $L^1$, take $f\in L^1\cap L^2$ and exhaust $\Om$ by finite-measure
	sets $K_m$.  Since
	$v\mapsto\norm{\one_{K_m}v}_1$ is convex and continuous on $L^2$, it is
	weakly lower semicontinuous.  This, together with the $L^1$-contractivity of $P_t^\eps$, yields
	\[
	\norm{\one_{K_m}P_tf}_1
	\le\liminf_{j\to\infty}
	\norm{\one_{K_m}P_t^{\eps_j}f}_1
	\le\norm f_1.
	\]
	Using monotone convergence and density, we conclude that $P_t$ extends to a contraction on $L^1$. 
	
	Next, we claim $P_t$ extends to a contraction on $L^\infty$. If this claim holds, then $P_t$ also extends to a  contraction on  $L^p$, $1<p<\infty,$ by
 	the Riesz-Thorin interpolation theorem
	and a  density argument. 

	By \eqref{eq:luwot}, 
	for each $t$,  the $L^2$-adjoint
	$(P_t^{\eps_j})^*$ converges to some 
	$Q_t$ in the w.o.t.~of $\cB(L^2)$. 
Applying the  weak
	lower-semicontinuity argument above shows that
	$Q_t$ extends to a contraction on $L^1$. 
By duality, the   adjoint $Q_t^*$  is a contraction on $L^\infty$, and agrees with $P_t$ on $L^2\cap L^\infty$. Hence $P_t$ extends to a contraction on $L^\infty$.  This proves the claim.

Finally,  
we show that this $u$ is a  weak solution to \eqref{PE1} in the sense of \eqref{DS}. Consider first the case $u_0\in L^2(\Om)$. 
Take any admissible test function $\varphi$.
Since $u^\eps$ are energy solutions, 
using the boundary conditions $\inn{b}{\nu}=0$ and
$\inn{a\nabla\varphi}{\nu}=0$ on $\di\Om$, eq.~\eqref{B5} gives
\begin{align}
\label{A9}
\int_0^\infty\!\!\int_\Om
u^\eps(\partial_t+L^*)\varphi\,dx\,dt
+\int_\Om u_0\varphi(\cdot,0)\,dx
=\eps\int_0^\infty\!\!\int_\Om
\nabla u^\eps\cdot\nabla\varphi\,dx\,dt.
\end{align}

Consider first the r.h.s.~of \eqref{A9}.  As $\eps\downarrow0$, we have
\begin{align*}
\eps	\left|\int_0^\infty\!\!\int_\Om
	\nabla u^\eps\cdot\nabla\varphi\right|
	&\le
	\sqrt\eps
	\left(\eps\int_0^\infty\|\nabla u^\eps\|_2^2\,dt\right)^{1/2}
	\|\nabla\varphi\|_{L^2(\Om\times I)}\\
	&\le C\sqrt\eps\,\|u_0\|_2
	\|\nabla\varphi\|_{L^2(\Om\times I)}\longrightarrow0.
\end{align*}
Next, we show that  the convergence \eqref{eq:luwot} implies, as $j\to\infty$,
\[
\left|\int_0^\infty \int_\Omega \del{u ^{\eps_j}-u }(\di_t + L^*)\varphi \right|\longrightarrow0.
\]
To see this, set \(h(t):=(\partial_t+L^*)\varphi(t)\).
Then
\(h\in L^1(0,T;L^2(\Om))\) for some $T>0$, since $\varphi$ has compact support. For a.e.~\(t\), applying
\eqref{eq:luwot} with \(g=h(t)\), gives
\[
\inn{u^{\eps_j}(t)-u(t)}{h(t)}_{L^2}\longrightarrow0.
\]
Moreover, by $L^2$-contractivity, we have
\[
\left|\inn{u^{\eps_j}(t)-u(t)}{h(t)}_{L^2}\right|
\le2\norm{u_0}_2\norm{h(t)}_2.
\]
Dominated convergence therefore gives $\int\int (u^{\eps_j}(t)-u(t)) h(t)\to0$, as desired.

Combining the above and sending $j\to\infty$ shows that 
$u$ is a weak solution  of \eqref{PE1}--\eqref{BC}, in the sense of \eqref{DS}, for $L^2$ initial data. 

Lastly, we show that $P_tu_0 $ satisfies \eqref{DS} for $L^p$ initial data.  
If $u_0 \in L^p(\Om)$ with $1\le p<\infty$, then using 
the $L^p$-contractivity of $P_t$ and approximating $u_0$ strongly by $u_{0,m}\in L^2\cap L^p$ gives \eqref{DS}. If $u_0\in L^\infty$, then take 
$u_{0,m}\in L^2\cap L^\infty$ with $\|u_{0,m}\|_\infty\le\|u_0\|_\infty$ and
$u_{0,m}\rightharpoonup^*u_0$ in $L^\infty$.  By a duality argument, we have   $P_tu_{0,m}\rightharpoonup^*P_tu_0$ in $L^\infty$ for each $t>0$.  Finally, pairing      with  
$h(t)=(\partial_t+L^*)\varphi(t)\in L^1$ and using dominated convergence gives \eqref{DS} for $p=\infty$.

This completes the proof of \propref{propA}.
\end{proof}

The direct exponential-deformation argument in \secref{sec22} gives the desired diffusion estimate at
the $L^1$ endpoint.  To obtain the $L^\infty$ endpoint with the same
exponent, we record the following adjoint compatibility property:

\begin{lemma}
	\label{lem:adjoint-compatible-viscosity}
		 {Let $P_t$ be as in \propref{propA}.} Then, given $t>0$, there exists a propagator $Q_r$, $0<r\le t$, generated by the operator
$$
L^\sharp(r)v
:=L(t-r)^* v =\operatorname{div}(a^\sharp(r)\nabla v)
+\langle b^\sharp(r),\nabla v\rangle+c^\sharp(r)v,$$
where
\[
a^\sharp(r)=a(t-r),\quad
b^\sharp(r)=-b(t-r),\quad
c^\sharp(r)=c(t-r)-\operatorname{div}b(t-r),
\]
such that $Q_t = P_t^*$ on $L^q$, $1\le q<\infty$. 

\end{lemma}

\begin{proof}

	For $\eps>0$, set $L_\eps=L+\eps\Delta$ as in \eqref{eq:approx} and
	\[
L_\eps^\sharp(r)
	:=L^\sharp(r)+\eps\Delta
	=L_\eps(t-r)^*.
	\]
	Let $P_t^\eps$ be the uniformly parabolic
	propagator generated by $L_\eps$, and let $Q_r^\eps$, $0<r\le t$, be the uniformly parabolic propagator generated by $L_\eps^\sharp$. Using the $L^p$--$L^q$ duality, it is straightforward to check that
	\[
	Q_t^\eps=(P_t^\eps)^*,\quad \eps>0.
	\]
Moreover, the assumptions on the coefficients $a,\,b,\,c$ ensure that $L_\eps^\sharp$ also satisfies the condition of \propref{propA}.
	
	Let $\eps_j\downarrow 0$ be as in \propref{propA}, with $P_t^{\eps_j}$ and $P_t$ satisfying \eqref{eq:luwot}. 
	Apply  \propref{propA} again to $Q_r^{\eps_j}$, after passing to a
	further subsequence, 
	we
	obtain a limit $Q_r\in \cB(L^2(\Omega))$ s.th.~
	$Q_r^{\eps_j}$ converges to $ Q_r$
	locally uniformly in $r$ in the weak operator topology of
	$\cB(L^2(\Omega))$, and    $Q_r$ is a propagator generated by $L^\sharp$. 
	
	It remains to show $Q_t= P_t^*$ on $L^q$ for  $1\le q<\infty$. 
Given $f,g\in L^2(\Omega)$, we compute
	\begin{align*}
		\langle Q_t f,g\rangle_{L^2}
		&=\lim_{j\to\infty}
		\langle(P_t^{\eps_j})^*f,g\rangle_{L^2}\\
		&=\lim_{j\to\infty}
		\langle f,P_t^{\eps_j}g\rangle_{L^2}\\
		&=\langle f,P_tg\rangle_{L^2}\\
		&=\langle P_t^*f,g\rangle_{L^2}.
	\end{align*}
	Hence $Q_t=P_t^*$ on $L^2$.  The claim now follows by continuity and the density of $L^q\cap L^2$ in $L^q$.  This completes the proof.  \end{proof}

	\section{Construction of separation functions}\label{secf}
\begin{proposition}\label{propB1}
	Let $\Om=\Rb^n$. Assume  $ \norm{a}_{L^\infty(\Rb^n\times I)}\le 1$ and $
	\beta:=\norm{\abs{b}+\abs{ \div a}}_{L^\infty(\Rb^n\times I)}<\infty .$
	Then, for any $\delta>0$,
	there exist $\k=\k(n,\beta,\delta)$
	and $\g=\g(n)$ s.th.~any  disjoint non-empty subsets $X,\,Y$ with $d_{XY}\ge\delta$  are $(\k,\g)$-separated.
\end{proposition}

\begin{proof}
	By taking closure if necessary, we assume in the remainder of this proof $X,\,Y$ are closed. 
	
	Write $
	\ell(x):=\dist(x,Y),$ and $
	d:=d_{XY}\ge \delta.
	$
	By Whitney's extension theorem  (see e.g., Theorem 2 on p.~171 of \cite{Ste}), there exists $
	\rho\in C(\Rb^n)\cap C^\infty(\Rb^n\setminus Y),$ $\rho|_Y=0,$
	such that, for some $c>1$ depending only on $n$ and all $x\in\Rb^n$,
	\begin{align}
	&	c^{-1}\ell(x)
		\le \rho(x)\le c\ell(x),\quad\text{ in }\Rb^n,\notag\\&
		\abs{\nabla\rho(x)}\le c,\quad
		\abs{\nabla^2\rho(x)}
		\le \frac{c}{\ell(x)},\quad\text{ in }\Rb^n\setminus Y.
		\label{eq:regularized-distance}
	\end{align}
	Choose
	$\vartheta\in C_c^\infty(\Rb)$ such that
	\[
	0\le\vartheta\le1,\qquad
	\vartheta=1\quad\text{on }\left[\frac1{2c},\frac1c\right],
	\qquad
	\supp\vartheta\subset\left(\frac1{4c},\frac2c\right).
	\]
	Define $
	G(s):=\int_{-\infty}^s\vartheta(u)\,du$
	and set
	\[
	\psi(x):=\frac{d}{c}
	G\left(\frac{\rho(x)}{d}\right).
	\]
	Since $G=0$ on $\left(-\infty,\frac1{4c}\right]$ and
	$\rho(x)\le c\ell(x)$, we have $
	\psi(x)=0$ whenever $
	\ell(x)\le \frac{d}{4c^2}.$
	Thus $\psi$ is smooth on all of $\Rb^n$.  Since, moreover, $G$ is constant on
	$\left[\frac2c,\infty\right)$, we have
	$\psi\in C_b^\infty(\Rb^n)$.
	
	We now show that this $\psi$ satisfies \eqref{kCond}--\eqref{gCond}.
	To begin with,  we compute
	\begin{align*}
		\nabla\psi
		=&\frac1c
		\vartheta\left(\frac{\rho }d\right)\nabla\rho,\\
		\nabla^2\psi
		=&\frac1c\left[
		\frac1d\vartheta'\left(\frac{\rho}{d}\right)
		\nabla\rho\otimes\nabla\rho
		+
		\vartheta\left(\frac{\rho}{d}\right)\nabla^2\rho
		\right].
	\end{align*}
	On the support of  $\vartheta(\rho/d)$, we have $
	\rho\ge\frac{d}{4c}$ and $
	\ell\ge\frac{\rho}{c}
	\ge\frac{d}{4c^2}.
	$
	Using \eqref{eq:regularized-distance} and the assumption $d\ge \delta$, we therefore obtain, in all of $\Rb^n$,
	\[
	\abs{\nabla\psi} \le1,\quad
	\norm{\nabla^2\psi}
	\le \frac{C(n)}{d}
	\le \frac{C(n)}{\delta}.
	\]
	Finally, since  $
	\div(a\nabla\psi)
	=(\div a)\cdot\nabla\psi+a:\nabla^2\psi,$ we have
	\begin{align*}
		\abs{\div(a\nabla\psi)}
		+\abs{\inn{b}{\nabla\psi}}
		&\le
		C(n)\norm{a}\,\abs{\nabla^2\psi}
		+\bigl(\abs{\div a}+\abs b\bigr)\abs{\nabla\psi}  \\
		&\le
		\frac{C(n)}{\delta}+\beta.
	\end{align*}
	Thus \eqref{kCond} holds with $\k:=
	\beta+\frac{C(n)}{\delta}$.
	
	Next, if $x\in X$,
	then $\ell(x)\ge d$ and hence $\rho(x)\ge d/c$.  Consequently,
	\begin{align*}
		\psi(x)
		&\ge \frac{d}{c}G\left(\frac1c\right)
		\ge \frac{d}{c}\int_{1/(2c)}^{1/c}\vartheta(u)\,du
		=\frac{d}{2c^2}.
	\end{align*}
	Since, moreover, $\psi(y)=0$ for $y\in Y$,
	it follows that $
	\psi(X,Y)\ge\frac{d}{2c^2}$, and therefore \eqref{gCond} holds with $
	\gamma:=\frac1{2c^2}.$
\end{proof}

\begin{proposition}
	Let $(M,g)$ be a complete non-compact Riemannian manifold of
	dimension $n\ge2$ satisfying, for some $K,\,i_0>0$,
	\[
	\abs{\Ric}\le(n-1)K,
	\qquad
	\inj(M,g)\ge i_0.
	\]
	Let
	$\Om$ be a  smooth domain  in $M$. 
	Assume  $ \norm a_{L^\infty(\Om\times I)}\le 1$ and $
	\beta:=\sup_{\Om\times I}
	\bigl(\abs{\div a}+\abs b\bigr)<\infty$.
	Then there exists $C:=C(n,K)>1$ such that the following holds.
	For any $\delta>C-1$, there exist $\k=\k(n,\beta,C)$, $\g=\g(C,\delta)$, s.th.~two non-empty subsets $X,\,Y\subset\Om$ are
	$(\k,\g)$-separated if
	$
	Y\subset\overline{B_r(o)}$ and $
	X\subset B_R(o)^\cp$ for $R>r>0$ and $o\in\Om$ satisfying  $R-r=d_{XY}\ge\delta$ and $B_{R+2}(o)\subset\Om$.
\end{proposition}

\begin{proof}
	The proof is similar to that of the preceding proposition.
	By  	Proposition~1.3 of \cite{RV},
	there exists $\rho\in C^\infty(M)$ such that, for
	$\ell(x):=d(o,x)$,
	\[
	\ell+1\le \rho\le \ell+C,\qquad
	\abs{\grad \rho}\le C,\qquad
	\norm{\nabla^2\rho}\le C.
	\]
	Choose $G\in C_b^\infty(\R)$ such that
	$0\le G'\le1$, $G'=1$ on $[r+C,R+1]$, $\supp G'\subset (r+C-1,R+2)$, and
	$\norm{G''}_\infty\le c_0$, where $c_0$ is an absolute
	constant. Choose $\psi=C^{-1}G\circ \rho\vert _{\bar\Om}$. Then
	$\psi\in C_b^\infty(\bar \Om)$ and 
	\[
	\sup_{\Om\times I}
	\left(
	\abs{\div(a\grad\psi)}
	+\abs{\inn{b}{\grad\psi}}
	\right)
	\le\beta+n({c_0C+1 } ).
	\]
	Consequently, \eqref{kCond} holds with $\k:=\beta+ n({c_0C+1} )$. Finally,
	the radial containments and the two-sided bounds for $\rho$ give  $\rho\le r+C$ on $Y$ and $\rho\ge R+1$ on
	$X$. Since, moreover, $G'=1$ on $[r+C,R+1]$ and $d_{XY}=R-r\ge \delta$, we find
	\[
	\psi(X,Y)
	\ge\frac{R-r-(C-1)}{C}
	\ge
	\frac1C\left(1-\frac{C-1}{\delta}\right)d_{XY}.
	\]
	Thus \eqref{gCond} holds with $
	\g:=\frac1C\left(1-\frac{C-1}{\delta}\right).$
\end{proof}

\bibliography{paraLE}
	\end{document}